\documentclass[]{amsart}

\usepackage{amsmath}
\usepackage{amsthm}
\usepackage{amsfonts}
\usepackage{amssymb}
\usepackage[shortlabels]{enumitem}
\usepackage{float}
\usepackage{mathtools}
\usepackage{csquotes}
\usepackage{changes}
\usepackage{refcount}

\usepackage[style=alphabetic,sorting=nyt,backend=bibtex8, maxbibnames=99]{biblatex}
\bibliography{reference.bib}

\usepackage{hyperref}
\hypersetup{colorlinks=true, citecolor=blue,urlcolor=black, linkcolor=blue}
\usepackage{graphicx}

\makeatletter
\newcommand{\oast}{\mathbin{\mathpalette\make@circled\ast}}
\newcommand{\make@circled}[2]{%
  \ooalign{$\m@th#1\smallbigcirc{#1}$\cr\hidewidth$\m@th#1#2$\hidewidth\cr}%
}
\newcommand{\smallbigcirc}[1]{%
  \vcenter{\hbox{\scalebox{0.77778}{$\m@th#1\bigcirc$}}}%
}
\makeatother

\newtheorem{theorem}{Theorem}
\newtheorem{maintheorem}{Theorem}

\newtheorem*{maincorollary1}{Corollary B}
\newtheorem*{maincorollary2}{Corollary C}
\newtheorem*{theorem*}{Theorem}

\newtheorem{lemma}[theorem]{Lemma}
\newtheorem{claim}{Claim}[theorem]
\newtheorem{subclaim}{Claim}[claim]
\theoremstyle{definition}

\newtheorem{question}{Question}
\newtheorem{problem}[question]{Problem}
\newtheorem*{question*}{Question}
\newtheorem{definition}[theorem]{Definition}

\theoremstyle{remark}

\begin{document}

\title{Quagmires and large Suslin forests}

\author{Lorenzo Notaro}
\address{University of Vienna, Institute of Mathematics, Kurt G\"{o}del Research Center, Kolingasse 14-16, 1090 Vienna, Austria}
\curraddr{}
\email{lorenzo.notaro@univie.ac.at}

\begin{abstract}
In 1972, Jech asked whether there exists a Suslin $(\omega_1, \omega_2)$-forest in the constructible universe. As reported by Jech, Laver gave a positive answer, but his proof was never published and appears no longer to be available. In 2015, Eskew introduced the combinatorial principle $W^*_\kappa(\lambda)$, a strengthening of Silver's principle, and used it to construct a coherent Suslin $(\kappa, \lambda)$-forest. He then asked whether the principle $W^*_{\kappa^+}(\kappa^{++})$ holds in $\mathsf{L}$ for every regular cardinal $\kappa$. We give an affirmative answer to Eskew's question, thereby also settling Jech's question.
\end{abstract}

\thanks{This research was funded in whole or in part by the Austrian Science Fund (FWF) \href{https://www.fwf.ac.at/en/research-radar/10.55776/ESP1829225}{10.55776/ESP1829225}. For open access purposes, the author has applied a CC BY public copyright license to any author accepted manuscript version arising from this submission.}
\subjclass[2020]{Primary 03E05, Secondary 03E45, 06E10}
\keywords{Suslin forest, coherent forest, morass, quagmire, Silver's principle, diamond principle}
\maketitle

\section{Introduction}

A forest is a structure that generalizes a tree. This notion was first introduced by Jech in \cite{MR325397} under the name  ``mess". The term ``forest" was later adopted by C. Wei{\ss} in \cite{Weiss2010}. In this work, we only consider binary forests.

\begin{definition}
Given two cardinals $\kappa < \lambda$, a $(\kappa, \lambda)$-forest\footnote{In Eskew's more general notation, this is a $(\kappa, \lambda, 2)$-forest \cite{MR3326048}.} is a collection $F$ of functions satisfying:
\begin{enumerate}
\item for every $f \in F$, $\mathrm{ran}(f) \subseteq \{0,1\}$,
\item $\{\mathrm{dom}(f) : f \in F\} = [\lambda]^{< \kappa}$,
\item for every $x,y \in [\lambda]^{< \kappa}$ with $x \subseteq y$,
\[
\{f \in F : \mathrm{dom}(f) = x\} = \{f \upharpoonright x : f \in F \text{ and } \mathrm{dom}(f) = y\}.
\]
\end{enumerate}
\end{definition}
Note that if $F$ is a $(\kappa, \lambda)$-forest and $\vec{x} = \langle x_\alpha : \alpha < \kappa\rangle$ is an injective sequence of ordinals below $\lambda$, then  $F_{\vec{x}} = \{f \in F : \mathrm{dom}(f) = \{x_\alpha : \alpha < \beta\} \text{ for some } \beta < \kappa\}$, ordered by extension, is a tree of height $\kappa$. 

The Suslin property admits a natural generalization from trees to forests:

\begin{definition}
A $(\kappa, \lambda)$-forest is \emph{Suslin} if the following conditions hold:
\begin{enumerate}
\item for every $x \in [\lambda]^{< \kappa}$, $|\{f \in F : \mathrm{dom}(f) = x\}| < \kappa$,
\item for every $A \in [\lambda]^\kappa$ and $f : A \rightarrow 2$, there exists $x \in [A]^{< \kappa}$ such that $f \upharpoonright x \not\in F$,
\item the forest $F$ ordered by reverse inclusion is $\kappa$-cc.
\end{enumerate}
\end{definition}

Given a Suslin $(\kappa, \lambda)$-forest $F$, any tree of the form $F_{\vec{x}}$ as above is $\kappa$-Suslin. Moreover, the Boolean completion of a Suslin $(\kappa, \lambda)$-forest is a $\kappa$-Suslin algebra, that is, a complete Boolean algebra which is both $\kappa$-cc and $\kappa$-distributive. By a well-known result of Solovay \cite[Theorem 30.20]{MR1940513}, every $\kappa$-Suslin algebra has cardinality at most $2^\kappa$. Hence, a Suslin $(\kappa, \lambda)$-forest exists only if $\lambda \le 2^\kappa$.  

In \cite{MR325397}, Jech proved that Suslin $(\omega_1, \lambda)$-forests can be forced for every prescribed uncountable cardinal $\lambda$. He then asked the following question:

\begin{question*}[Jech]
Does $\mathsf{V = L}$ imply the existence of a Suslin $(\omega_1, \omega_2)$-forest?
\end{question*}

As reported by Jech  \cite{MR325397, MR506523, MR1940513}, Laver answered this question in the positive in an unpublished work. Unfortunately, Laver's proof appears to be no longer available \cite{MR3326048}. 

More recently, Eskew \cite{MR3326048} studied \emph{coherent} forests. A $(\kappa^+, \lambda)$-forest $F$ is coherent if for every $f,g \in F$, $|\{\alpha \in \mathrm{dom}(f) \cap \mathrm{dom}(g) : f(\alpha) \neq g(\alpha)\}| < \kappa$. While working on Jech's question, Eskew introduced the combinatorial principle $W^*_\kappa(\lambda)$ (see Section~\ref{sec:silver}), a strong form of Silver's principle $W_\kappa(\lambda)$ with built-in diamond, and proved:
\begin{theorem}[Eskew]
Given two cardinals $\kappa < \lambda$ with $\kappa$ regular, if $W^*_{\kappa^+}(\lambda)$ holds, then there exists a coherent Suslin $(\kappa^+, \lambda)$-forest closed under ${<}\kappa$-modifications.
\end{theorem}
Here, a forest $F$ is \emph{closed under ${<}\kappa$-modifications} if whenever $f \in F$ and $g: \mathrm{dom}(f) \rightarrow 2$ differs from $f$ on fewer than $\kappa$ coordinates, then $g \in F$. 

Relative to a Mahlo cardinal, Eskew proved that $W^*_{\kappa^+}(\lambda)$ is consistent for arbitrarily large $\lambda$. He then asked whether $W^*_{\kappa^+}(\kappa^{++})$ holds under $\mathsf{V= L}$ for every regular cardinal $\kappa$. In this work, we answer Eskew's question in the positive.

\begin{maintheorem}\label{thm:main}
For every infinite cardinal $\kappa$, if $\kappa^{< \kappa} = \kappa$, $\diamondsuit(E^{\kappa^+}_\kappa)$ holds and there is a $(\kappa^+, 1)$-morass, then  $W^*_{\kappa^+}(\kappa^{++})$ holds.
\end{maintheorem}

In fact, the full combinatorial power of a morass is not needed. The proof uses only a fragment of it, essentially the structure isolated by Burgess under the name of a \emph{quagmire} \cite{MR555420}---see Definition~\ref{def:quagmire}. A quagmire naturally arises from the tree structure inherent in a simplified morass (Theorem~\ref{thm:quagmire}). We formulate the result in this greater generality in Section~\ref{sec:main}. 

By combining our result with Eskew's construction, we give a strong positive answer to Jech's question.

\begin{maincorollary1}
If $\mathsf{V = L}$, then for every infinite regular cardinal $\kappa$ there exists a coherent Suslin $(\kappa^+, \kappa^{++})$-forest closed under ${<}\kappa$-modifications.
\end{maincorollary1}

Theorem~\ref{thm:main} also has the following consequence for Suslin algebras. Given a coherent $(\kappa^+, \lambda)$-forest which is closed under ${<}\kappa$-modifications, it is not hard to see that its Boolean completion is homogeneous (cf. the Solovay–Koppelberg theorem \cite[Theorem 18.4.1]{MR991595} for a more general result). Recall that a Boolean algebra $B$ is homogeneous if $B \upharpoonright a \cong B$ for every $a \in B^+$. This yields the following corollary.
\begin{maincorollary2}
If $\mathsf{V = L}$, then for every infinite regular cardinal $\kappa$ there exists a homogeneous $\kappa^+$-Suslin algebra of cardinality $\kappa^{++}$.
\end{maincorollary2}
For $\kappa = \omega$, the existence of a homogeneous (and chain-homogeneous) Suslin algebra of cardinality $\aleph_2$ that has no complete generating set of cardinality $\aleph_1$ was previously established by Scharfenberger-Fabian from $\diamondsuit^+$ via a very different construction \cite{MR2853724}.

The paper is organized as follows. In Section~\ref{sec:prel}, we introduce the preliminary notions and discuss both Silver's and Eskew's combinatorial principles.  In Section~\ref{sec:quagmire}, we introduce a slight generalization of Burgess’s notion of a quagmire and show how it arises by retaining essentially only the tree structure underlying a simplified morass. In Section~\ref{sec:main}, we prove the more general version of Theorem~\ref{thm:main}, formulated in terms of quagmires (Theorem~\ref{thm:mains}). Finally, in Section~\ref{sec:questions}, we record some open questions.

\section{Preliminaries}\label{sec:prel}

\subsection{Notation}
The monograph \cite{MR1940513} is our reference for all classical definitions and notation in set theory.

Given two cardinals $\kappa < \lambda$ with $\kappa$ regular, $E^\lambda_\kappa$ denotes the set of all ordinals below $\lambda$ of cofinality $\kappa$. Given a set $X$ and a cardinal $\kappa$, we denote by $[X]^{<\kappa}$ and $[X]^\kappa$ the families of all subsets of $X$ of cardinality $< \kappa$ and $\kappa$, respectively. The set $[X]^{\le \kappa}$ is just $[X]^{<\kappa} \cup [X]^\kappa$.

A \emph{tree} $(T, \le)$ is a poset such that, for each $x \in T$, the set $\{y \in T \mid y < x\}$ is well-ordered by $\le$. If $x \in T$, the \emph{height} of $x$ in $T$, denoted by $\mathrm{ht}(x)$, is the order type of $\{y \in T \mid y < x\}$. The height of a tree $T$ is defined as $\mathrm{ht}(T) = \sup \{\mathrm{ht}(x)+1 : x \in T\}$. For each ordinal $\alpha < \mathrm{ht}(T)$, the $\alpha$-th \emph{level} of $T$, denoted by $T_\alpha$, is the set of all nodes of $T$ of height $\alpha$. Moreover, we denote the subtree $\bigcup_{\beta < \alpha} T_\beta$ by $T \upharpoonright \alpha$.

Given $\alpha < \beta < \mathrm{ht}(T)$, the map $\pi_{\alpha \beta} : T_\beta \rightarrow T_\alpha$ sends each node of $T_\beta$ to its unique predecessor of height $\alpha$. The map $\pi_\alpha$ is the union of the maps $\pi_{\alpha \beta}$ for $\alpha < \beta < \mathrm{ht}(T)$.

For an uncountable cardinal $\kappa$, a $\kappa$-tree is a tree of height $\kappa$ whose levels have cardinality $< \kappa$. A $\kappa$-Kurepa tree is a $\kappa$-tree that has at least $\kappa^+$ cofinal branches. In this paper, it will be notationally convenient to identify the cofinal branches of a $\kappa$-Kurepa tree with an additional top level. More precisely, we regard a $\kappa$-Kurepa tree as a tree $T$ of height $\kappa+1$ such that $T \upharpoonright \kappa$ is a $\kappa$-tree, $|T_\kappa| > \kappa$, and for any two distinct $x,y \in T_\kappa$, there exists $\alpha < \kappa$ such that $\pi_\alpha(x) \neq \pi_\alpha(y)$.

\subsection{Silver's and Eskew's principles}\label{sec:silver}
Silver's principle $W$ is a morass-like combinatorial principle isolated by Silver. It asserts the existence of a Kurepa tree augmented by a combinatorial structure used to ``capture" small subsets of the top level at lower levels of the tree. For a detailed introduction to Silver's principle, we refer the reader to \cite{MR673791, MR823780}.

\begin{definition}
Silver's principle $W_\kappa(\lambda)$ asserts that there exist a $\kappa$-Kurepa tree $T$ with $|T_\kappa| = \lambda$ and a sequence $\langle W_\alpha : \alpha < \kappa\rangle$ such that:
\begin{enumerate}
\item for every $\alpha < \kappa$, $W_\alpha$ is a subset of $\mathcal{P}(T_\alpha)$ of cardinality  $<\kappa$,
\item for every $X \in [T_{\kappa}]^{<\kappa}$, there exists $\alpha < \kappa$ such that $\pi_\beta[X] \in W_\beta$ for every $\beta$ with $\alpha \le \beta < \kappa$.
\end{enumerate}
\end{definition}

As reported by Jech in \cite{MR506523}, Laver proved that the existence of a Suslin $(\omega_1, \omega_2)$-forest follows from $W_{\omega_1}(\omega_2) + \diamondsuit$. However, as mentioned in the introduction, the proof appears to have been lost. We return to this in Section~\ref{sec:questions}.

Before stating Eskew's strengthening of Silver's principle, it is useful to introduce the following terminology. Let $T$ and $\langle W_\alpha : \alpha < \kappa\rangle$ witness $W_\kappa(\lambda)$. Given $\alpha < \gamma \le \kappa$ and a set $A \subseteq T_\gamma$, we say that $A$ is \emph{captured at $\alpha$} if, for every $\beta$ with $\alpha \le \beta < \gamma$, $\pi_\beta[A] \in W_\beta$ and $\pi_\beta \upharpoonright A$ is injective. We say that $A$ is \emph{captured below $\gamma$} if it is captured at some $\alpha<\gamma$.

Eskew introduced the following strengthening of Silver's principle in order to construct a coherent Suslin $(\kappa^+, \lambda)$-forest.

\begin{definition}
Given two infinite cardinals $\kappa< \lambda$, where $\kappa = \theta^+$ for some regular cardinal $\theta$, the principle $W^*_{\kappa}(\lambda)$ asserts that there exist a $\kappa$-Kurepa tree $(T, <)$ with $|T_{\kappa}| = \lambda$, a sequence $\langle W_\alpha : \alpha < \kappa\rangle$, a stationary set $S \subseteq \kappa$, and a sequence $\langle A_\alpha : \alpha < \kappa\rangle$ with $A_\alpha \subseteq W_\alpha^2$, such that:
\begin{enumerate}
\item $(T, <)$ and $\langle W_\alpha : \alpha < \kappa\rangle$ witness $W_{\kappa}(\lambda)$,
\item for every $\alpha < \kappa$, $W_\alpha$ contains $\{x\}$ for every $x \in T_\alpha$ and it is a $\theta$-complete Boolean subalgebra of $\mathcal{P}(T_\alpha)$,
\item for every $\alpha \in S$, the set $\{X \in W_\alpha : X \text{ is captured below } \alpha\}$ is closed under ${<}\theta$-unions and under taking subsets that belong to $W_\alpha$,
\item if $f:\kappa \rightarrow ([T_{\kappa}]^{\le \theta})^2$ is such that $\bigcup_{\alpha < \kappa} (f_0(\alpha) \cup f_1(\alpha))$ has cardinality $\kappa$ and $\langle b_\mu : \mu < \kappa\rangle$ enumerates it, then the set of all $\alpha \in S$ satisfying the following conditions is stationary:
\begin{enumerate}
\item $\{b_\mu : \mu < \alpha\}$ is captured at $\alpha$,
\item if $X \subseteq \{\pi_\alpha(b_\mu) : \mu < \alpha\}$ is captured below $\alpha$, then \[\sup \{\mu < \alpha: \pi_{\alpha}(b_\mu) \in X\} < \alpha,\]
\item $\{\langle \pi_\alpha[f_0(\eta)], \pi_\alpha[f_1(\eta)]\rangle : \eta < \alpha\} = A_\alpha$.
\end{enumerate} 
\end{enumerate}
\end{definition}

So Eskew's $W^*$ principle is a strengthening of Silver's principle $W$,  augmented by a diamond-like sequence $\langle A_\alpha : \alpha < \kappa\rangle$ used to guess, stationarily often, the projections of $\kappa$-sequences of pairs of elements of $[T_\kappa]^{< \kappa}$. Eskew proved the following consistency result \cite{MR3326048}.
\begin{theorem}[Eskew]
Suppose that $\kappa$ is a Mahlo cardinal and $\mu < \kappa$ is regular. If $G$ is $V$-generic for $\mathrm{Col}(\mu, {<}\kappa) * \mathrm{Add}(\kappa, 1)$, then $V[G]$ satisfies $W^*_\kappa(2^\kappa)$.
\end{theorem}

He also asked whether $W^*_\kappa(\lambda)$ can be forced without large cardinals and whether it can be forced in a cardinal-preserving way. Our Theorem~\ref{thm:main} gives a positive answer to both questions when $\lambda = \kappa^+$, under the hypotheses $2^{\theta} = \kappa$ and $2^{<\theta} = \theta$, where $\kappa = \theta^+$. Indeed, for every uncountable regular cardinal $\kappa$ satisfying $2^{<\kappa} = \kappa$, there exists a $\kappa$-closed and $\kappa^{+}$-cc forcing that adds a simplified $(\kappa, 1)$-morass \cite{MR736620}.

\section{Morasses and quagmires}\label{sec:quagmire}

As mentioned in the introduction, we only need a fragment of a morass in order to prove Theorem~\ref{thm:main}. What we need is a combinatorial structure called a quagmire, introduced by Burgess in \cite{MR555420}. Such a structure essentially retains only the tree structure inherent in simplified morasses (see the proof of Theorem~\ref{thm:quagmire} below). What follows is a slight generalization of Burgess's original definition.

\begin{definition}\label{def:quagmire}
Given an uncountable cardinal $\kappa$ and a set $S \subseteq \mathrm{Lim}(\kappa)$, a\linebreak \emph{$(\kappa, S)$-quagmire} is a quadruple $(T, {<}, {\lhd}, Q)$ such that:
\begin{enumerate}[label={(Q\arabic*)},  start=0]
\item 
\begin{enumerate}
\item $(T, <)$ is a $\kappa$-Kurepa tree,
\item $\lhd$ is a binary relation on $T$ whose restriction to each level of $T$ is a strict linear order; nodes of distinct heights are $\lhd$-incomparable,
\item the function $Q$ is a partial function from $T^2$ to $T$ whose domain consists of the pairs $(x, y) \in T^2$ such that $\mathrm{ht}(x) < \mathrm{ht}(y)$ and $x \lhd \pi_{\mathrm{ht}(x)}(y)$,
\end{enumerate}

\item if $(x, y) \in \mathrm{dom}(Q)$, then $x < Q(x, y) \lhd y$,
\item if $(x, y) \in \mathrm{dom}(Q)$ and $y < z$, then $Q(Q(x, y), z) = Q(x,z)$,
\item if $(y, z) \in \mathrm{dom}(Q)$ and $x \lhd y$, then $Q(x, Q(y, z)) = Q(x, z)$,
\item if $x \lhd y$ and $\mathrm{ht}(x) \in S \cup \{\kappa\}$, then there exists $\alpha < \mathrm{ht}(x)$ such that $\pi_\alpha(x) \lhd \pi_\alpha(y)$ and $x = Q(\pi_\alpha(x), y)$.
\end{enumerate}
\end{definition}

A $(\kappa, \emptyset)$-quagmire is precisely a $\kappa$-quagmire in Burgess's original terminology. 
For convenience, we extend $Q$ by setting $Q(x,y)=y$ whenever $x<y$, and $Q(x,y)=x$ whenever $x\lhd y$. Note that (Q2)-(Q4) are unaffected by this extension.

Next, we show that every simplified $(\kappa,1)$-morass gives rise to a $(\kappa,\mathrm{Lim}(\kappa))$-quagmire. This follows immediately from the tree structure naturally associated with a simplified morass. We assume some familiarity with the definition of simplified morasses (see, e.g., \cite{MR736620}  and \cite[Ch. VIII, \S 4]{MR750828}). By a well-known result of Velleman, who introduced simplified morasses, the existence of a simplified $(\kappa, 1)$-morass is equivalent to the existence of a $(\kappa, 1)$-morass \cite{MR736620}.

\begin{theorem}\label{thm:quagmire}
For every uncountable regular cardinal $\kappa$, if there is a simplified $(\kappa, 1)$-morass, then there is a $(\kappa, \mathrm{Lim}(\kappa))$-quagmire.
\end{theorem}
\begin{proof}

Let $(\langle \theta_\alpha : \alpha \le \kappa\rangle, \langle \mathcal{F}_{\alpha, \beta} : \alpha < \beta \le \kappa\rangle)$ be a simplified $(\kappa, 1)$-morass. 
Let $T = \bigcup_{\alpha \le \kappa} \{\alpha\} \times \theta_\alpha$. Given $s,t \in T$ with $s = (\alpha, \mu)$ and $t = (\beta, \nu)$, let $s \prec t$ if and only if $\alpha < \beta$ and there exists $f \in \mathcal{F}_{\alpha, \beta}$ such that $f(\mu) = \nu$. 

The binary relation $\prec$ is a tree ordering and $T_\alpha = \{\alpha\} \times \theta_\alpha$ for every $\alpha \le \kappa$ \cite[Theorem 3.1]{MR736620}. In particular, $(T, \prec)$ is a $\kappa$-Kurepa tree.

Moreover, suppose that $s=(\alpha,\mu)\prec t=(\beta,\nu)$, and let $f\in\mathcal F_{\alpha,\beta}$ witness $s\prec t$. Then the restriction $f\upharpoonright(\mu+1)$ is uniquely determined (i.e., it does not depend on the choice of $f$) \cite[Lemma 3.2]{MR736620}. Hence, the map $\pi_{st} \coloneqq f \upharpoonright (\mu +1)$ is well-defined, and the following properties hold \cite[Theorem 3.3]{MR736620}:
\begin{enumerate}
\item if $t_0 \prec t_1 \prec t_2$, then $\pi_{t_0 t_2} = \pi_{t_1 t_2} \circ \pi_{t_0 t_1}$,
\item if $s \prec t$, where $s = (\alpha, \mu)$ and $t = (\beta, \nu)$, then for $\tau < \mu$, letting $s' = (\alpha, \tau)$ and $t' = (\beta, \pi_{st}(\tau))$, we have $s' \prec t'$ and $\pi_{s't'} = \pi_{st} \upharpoonright (\tau + 1)$,
\item if $\alpha \le \kappa$ is a limit ordinal and $t = (\alpha, \nu)$, then $\nu + 1 = \bigcup \{\mathrm{ran}(\pi_{st}) : s \prec t\}$.
\end{enumerate}

Now it is easy to define $\lhd$ and $Q$ so that $(T, \prec, \lhd, Q)$ is a $(\kappa, \mathrm{Lim}(\kappa))$-quagmire. First, for every $\alpha \le \kappa$ and $\nu, \mu < \theta_\alpha$, let $(\alpha, \nu) \lhd (\alpha, \mu)$ if and only if $\nu < \mu$. Next, suppose that $s \prec t$, where $s = (\alpha, \mu)$ and $t = (\beta, \nu)$. For $s' = (\alpha, \tau)$ with $\tau < \mu$, let 
\[
Q(s', t) \coloneqq (\beta, \pi_{st}(\tau)).
\]

Property (Q1) follows directly from (2) and from the fact that the maps $\pi_{st}$ are order-preserving. Properties (Q2), (Q3), and (Q4) follow directly from (1), (2), and (3), respectively.
\end{proof}

Finally, we record some basic properties of quagmires that will be useful in the next section.

\begin{lemma}\label{lemma:basicquag}
Let $(T, <, \lhd, Q)$ be a $(\kappa, S)$-quagmire. Then the following hold:
\begin{enumerate}[label={\upshape (\arabic*)}]
\item $|T_\kappa| = \kappa^+$ and every element of $[T_{\kappa}]^{\le \kappa}$ has an $\lhd$-upper bound;
\item for every $x \in T$ and every $\alpha < \mathrm{ht}(x)$, the map $Q(\cdot, x) \upharpoonright T_\alpha$ is an order-embedding;
\item for every $(x, y) \in \mathrm{dom}(Q)$ and $\alpha$ with $\mathrm{ht}(x) < \alpha < \mathrm{ht}(y)$, $\pi_\alpha(Q(x, y)) = Q(x, \pi_\alpha(y))$;
\item for every $x, y$ and $\alpha < \beta < \mathrm{ht}(x)$, if $x = Q(\pi_\alpha(x), y)$, then $x = Q(\pi_\beta(x), y)$.
\end{enumerate}
\end{lemma}
\begin{proof}
$(1)$: It suffices to prove that every $\lhd$-initial segment of $T_\kappa$ has cardinality at most $\kappa$. Indeed, once this is proved, it directly follows from the fact that $T$ is a $\kappa$-Kurepa tree that $|T_\kappa| = \kappa^+$ and that no subset of $T_\kappa$ of cardinality at most $\kappa$ is $\lhd$-cofinal in $T_\kappa$.

Fix $x\in T_\kappa$. We show that its $\lhd$-initial segment has cardinality at most $\kappa$. For any $y \lhd x$, there exists, by (Q4), an ordinal $\alpha_y < \kappa$ such that $\pi_{\alpha_y}(y) \lhd \pi_{\alpha_y}(x)$ and $y = Q(\pi_{\alpha_y}(y), x)$. Hence, this observation together with (Q1) gives
\[
\{y : y \lhd x\} = \bigcup_{\alpha < \kappa} \{Q(z, x) : z \lhd \pi_\alpha(x)\}.
\]
Since the set on the right-hand side clearly has cardinality at most $\kappa$, we are done.

$(2)$: Pick $y,z$ with $z \lhd y \trianglelefteq \pi_\alpha(x)$. By (Q3), $Q(z, x) = Q(z, Q(y, x))$. By (Q1), $Q(z, Q(y, x)) \lhd Q(y, x)$. Hence, $Q(z, x) \lhd Q(y, x)$.

$(3)$: By (Q2), $Q(x, y) = Q(Q(x, \pi_\alpha(y)), y)$. Moreover, by (Q1), we have \[
Q(x, \pi_\alpha(y)) < Q(Q(x, \pi_\alpha(y)), y).
\]
Hence, $Q(x, \pi_\alpha(y)) < Q(x, y)$ or, equivalently, $\pi_\alpha(Q(x, y)) = Q(x, \pi_\alpha(y))$.

$(4)$: By (Q2), 
\[
Q(\pi_\alpha(x), y) = Q(Q(\pi_\alpha(x), \pi_\beta(y)), y).
\]
 By (3), 
\[
Q(\pi_\alpha(x), \pi_\beta(y)) = \pi_\beta(Q(\pi_\alpha(x), y)) = \pi_\beta(x).
\]
 Hence, $Q(\pi_\alpha(x), y) = Q(\pi_\beta(x), y)$.
\end{proof}

\section{Main result}\label{sec:main}

This section is devoted to the proof of the following theorem. 

\begin{theorem}\label{thm:mains}
Given an infinite cardinal $\kappa$ with $\kappa^{< \kappa} = \kappa$, if there is a stationary $S \subseteq E^{\kappa^+}_\kappa$ such that $\diamondsuit(S)$ holds and a $(\kappa^+, S)$-quagmire exists, then $W^*_{\kappa^+}(\kappa^{++})$ holds.
\end{theorem}

Theorem~\ref{thm:main} follows directly from Theorems~\ref{thm:quagmire} and \ref{thm:mains}. But before proving Theorem~\ref{thm:mains}, we need the following routine lemma (cf. \cite[Lemma 2.2]{MR3692231}).

\begin{lemma}\label{lemma:diamond}
Given an infinite cardinal $\kappa$ with $\kappa^{<\kappa} = \kappa$ and a stationary $S \subseteq E^{\kappa^+}_\kappa$, the following are equivalent:
\begin{enumerate}[label={\upshape (\arabic*)}]
\item $\diamondsuit(S)$,
\item There exists a sequence $\langle \Omega_\alpha : \alpha < \kappa^{+}\rangle$ such that for every cardinal $\mu > \kappa^+$, parameter $p \in H(\mu)$, and subset $\Omega \subseteq H(\kappa^{+})$, there exists an elementary submodel $N \prec H(\mu)$ satisfying:
\begin{enumerate}[label={\upshape (\alph*)}]
\item ${}^{<\kappa} N \subseteq N$,
\item $p \in N$,
\item $\kappa^+ \cap N \in S$,
\item $\Omega_{\kappa^+ \cap N} = \Omega \cap N$.
\end{enumerate}
\end{enumerate}
\end{lemma}

\begingroup
\setcounter{claim}{0}
\renewcommand{\theclaim}{\getrefnumber{thm:mains}.\arabic{claim}}

\begin{proof}[Proof of Theorem~\ref{thm:mains}]
Denote $\kappa^+$ by $\lambda$. Let $(T, <, \lhd, Q)$ be a $(\lambda, S)$-quagmire and let $\langle \Omega_\alpha : \alpha < \lambda\rangle$ witness Lemma~\ref{lemma:diamond}(2). We can assume without loss of generality that $T \in H(\lambda^{++})$ and $T \upharpoonright \lambda \subseteq \lambda$. 

For every $\alpha < \lambda$, enumerate $\mathcal{P}(T_\alpha)$ as $\langle X^\alpha_\xi : 0 < \xi < \lambda\rangle$. We intentionally leave $X^\alpha_0$ undefined for now. The indices $\xi > 0$ will provide ordinary bookkeeping for the subsets of the levels of $T$, while the index $\xi = 0$ is reserved for the information supplied by our diamond sequence.
 
We first use $\Omega_\alpha$ to decode possible guesses for nodes of height $\alpha$. For every $\alpha \in S$ and $\mu < \alpha$, if there exists $x \in T_\alpha$ such that
\[
\big\{y \in T \upharpoonright \alpha : (\mu, y) \in \Omega_\alpha\}
\]
 is a $<$-cofinal subset of $\{z \in T : z < x\}$, then, by (Q4) (using $\alpha \in S$), such an $x$ is unique, and we denote it by $x_\mu^\alpha$. Otherwise, choose $x_\mu^\alpha \in T_\alpha$ arbitrarily.

Now we are ready to define the set $X^\alpha_0 \subseteq T_\alpha$ for every $\alpha < \lambda$ by induction on $\alpha$. If $\alpha \in S$ and satisfies the following property $(\dagger)$, then we let $X^\alpha_0 = \{x_\mu^\alpha : \mu < \alpha\}$:
\begin{center}
  \begin{tabular}{@{}r@{\qquad}l@{}}
    \raisebox{-.5\height}{\textnormal{($\dagger$)}} &
    \begin{minipage}[c]{\dimexpr\linewidth-5em\relax}
      There exists an elementary submodel
      $N \prec H(\lambda^{++})$ and a set
      $Y \in [T_{\lambda}]^{\lambda}$ such that the following conditions hold:
      \begin{enumerate}
        \item ${}^{< \kappa} N \subseteq N$,
        \item $\alpha=N\cap\lambda$,
        \item $T,{<},{\lhd},Q,
          \langle X^\beta_\xi:\beta<\lambda,\ 0<\xi<\lambda\rangle\in N$,
        \item $X^\beta_0\in N$ for every $\beta<\alpha$,
        \item $Y\in N$,
        \item $\{x_\mu^\alpha:\mu<\alpha\}=\pi_\alpha[Y\cap N]$.
      \end{enumerate}
    \end{minipage}
  \end{tabular}
\end{center}
If $\alpha \not\in S$ or does not satisfy $(\dagger)$, then we simply let $X^\alpha_0 = \emptyset$. This completes the definition of the sequence $\langle X^\alpha_0 : \alpha < \lambda\rangle$. Intuitively, we let $X^\alpha_0$ be the set of guessed nodes of height $\alpha$ precisely when this set arises as the projection of a set of cofinal branches whose cardinality is, morally, $\lambda$. The sets $X^\alpha_0$ will play a crucial role in ensuring properties (4a) and (4b) of $W^*_{\lambda}(\lambda^+)$.

Next we use the map $Q$ to lift $\trianglelefteq$-initial segments of the sets $X^\beta_\xi$ to higher levels of $T$. Given $\alpha \le \lambda$, $H \in [T_\alpha]^{<\kappa}$, $\beta < \alpha$, and $\xi < \lambda$, let
\[
D(H, \beta, \xi) \coloneqq \big\{Q(\bar{y}, x) : x \in H, \bar{y} \trianglelefteq \pi_\beta(x) \text{ and } \bar{y}\in X^\beta_\xi\big\},
\]
where $\trianglelefteq$ is the reflexive closure of $\lhd$---recall our extension of $Q$ (see Section~\ref{sec:quagmire}). Thus, $D(H, \beta, \xi)$ consists of the lifts, via $Q$, of the elements of $X^\beta_\xi$ lying $\trianglelefteq$-below the $\beta$-projections of the elements of $H$. If $H = \{x\}$, then we simply write $D(x, \beta, \xi)$ instead of $D(\{x\}, \beta, \xi)$. Clearly, $D(H, \beta, \xi) = \bigcup_{x \in H} D(x, \beta, \xi)$.

We use frequently the following observation: for $x \in T_\alpha$, (Q1) implies that whenever $y \in D(x, \beta, \xi)$, the witness $\bar{y}$ must be equal to $\pi_\beta(y)$. In other words, for every $x,y \in T_\alpha$, $y \in D(x, \beta, \xi)$ if and only if $\pi_\beta(y) \trianglelefteq \pi_\beta(x)$ and $\pi_\beta(y) \in X^\beta_\xi$ and $y = Q(\pi_\beta(y), x)$.

For every $\alpha < \lambda$ let
\[
V_\alpha \coloneqq \big\{D(H, \beta, \xi) : H \in [T_\alpha]^{<\kappa}  \text{ and } \beta, \xi  < \alpha\big\}.
\]
Let $W_\alpha$ be the closure of $V_\alpha \cup \{ X^\alpha_0\} \cup \{\{x\} : x \in T_\alpha\}$ under complements (in $T_\alpha$) and ${<}\kappa$-unions. In other words, $W_\alpha$ is the $\kappa$-complete Boolean subalgebra of $\mathcal{P}(T_\alpha)$ generated by $V_\alpha \cup \{ X^\alpha_0\} \cup \{\{x\} : x \in T_\alpha\}$.  Note that $W_\alpha$ has cardinality at most $\kappa$, since we are assuming $\kappa^{<\kappa} = \kappa$.

Now we use $\Omega_\alpha$ again  to decode possible guesses for pairs of elements of $W_\alpha$. For each $\alpha < \lambda$ we define a set $A_\alpha \subseteq W_\alpha^2$. If $\alpha \not\in S$, then let $A_\alpha = \emptyset$. Otherwise, let $A_\alpha$ be the set of all pairs $(u_0, u_1) \in W_\alpha^2$ such that there exists $\beta < \alpha$ for which
\[
u_i = \big\{x_\gamma^\alpha  : (\beta, \gamma, i) \in \Omega_\alpha\big\} \text{ for } i = 0,1.
\]

Finally, let $S' \subseteq S$ be the set of all ordinals $\alpha \in S$ such that there exists an elementary submodel $N \prec H(\lambda^{++})$ with
\begin{enumerate}[label={(\alph*)}]
\item ${}^{<\kappa}N \subseteq N$,
\item $\alpha = N \cap \lambda$,
\item $T, {<}, {\lhd},Q, \langle X^\beta_\gamma:\beta, \gamma <\lambda\rangle\in N$.
\end{enumerate}
Note that $S \setminus S'$ is nonstationary. Note also that for every $\alpha \in S'$, with witness $N$, $T \upharpoonright \alpha = (T \upharpoonright \lambda) \cap N$. Indeed, $\kappa \in N$ and $\kappa \subseteq N$ and $T_\beta$ has cardinality at most $\kappa$ for every $\beta < \lambda$. 

The rest of the proof consists in showing that $(T,<), \langle W_\alpha : \alpha < \lambda\rangle$, $S'$, and $\langle A_\alpha : \alpha < \lambda\rangle$ witness $W^*_{\lambda}(\lambda^+)$. Clause (1) of  $W^*_{\lambda}(\lambda^+)$ is proved in Claim~\ref{claim:1}; (2) follows directly from the definition of $W_\alpha$; (3) is proved in Claim~\ref{claim:9}; (4) follows from Claims~\ref{claim:11}, \ref{claim:12}, and \ref{claim:13}.

\begin{claim}\label{claim:1}
$(T, <)$ and $\langle W_\alpha : \alpha < \lambda\rangle$ witness $W_\lambda(\lambda^+)$.
\end{claim}
\begin{proof}
The argument can be found in \cite[\S 3]{MR555420} and \cite[Theorem 4]{MR823780}. We have already noted that $W_\alpha$ is a subset of $\mathcal{P}(T_\alpha)$ of cardinality at most $\kappa$. So let us pick some $A \in [T_{\lambda}]^{\le \kappa}$. By Lemma~\ref{lemma:basicquag}(1), $A$ has an $\lhd$-upper bound $\bar{b}$. By (Q4), for every $b \in A$, there exists $\gamma_b < \lambda$ such that $\pi_{\gamma_b}(b) \lhd \pi_{\gamma_b}(\bar{b})$ and $b = Q(\pi_{\gamma_b}(b), \bar{b})$. Let 
\[
\gamma = \sup_{b \in A} \gamma_b.
\]
 As $A$ has cardinality less than $\lambda$, $\gamma < \lambda$.

Let $0 < \xi < \lambda$ be such that $\pi_\gamma[A] = X^\gamma_\xi$. We are done once we show that $\pi_\beta[A] = D(\pi_\beta(\bar{b}), \gamma, \xi)$ for every $\beta$ with $\max(\gamma, \xi) < \beta < \lambda$. So pick one such $\beta$.

For every $b \in A$,
\begin{equation}\label{eq:start}
\begin{split}
Q(\pi_\gamma(b), \pi_\beta(\bar{b})) &= \pi_\beta(Q(\pi_\gamma(b), \bar{b}))\\&= \pi_\beta(b),
\end{split}
\end{equation}
where the first equality follows from Lemma~\ref{lemma:basicquag}(3), while the second follows from the choice of $\gamma$ and Lemma~\ref{lemma:basicquag}(4).

Now pick $x \in D(\pi_\beta(\bar{b}), \gamma, \xi)$. By the definition of $D$ and the choice of $\xi$, there exists $b \in A$ such that $\pi_\gamma(b) = \pi_\gamma(x)$ and $x = Q(\pi_\gamma(x), \pi_\beta(\bar{b}))$. So, by \eqref{eq:start}, $x = Q(\pi_\gamma(b), \pi_\beta(\bar{b})) = \pi_\beta(b)$, which means $x \in \pi_\beta[A]$.

Conversely, if $b \in A$, then, by \eqref{eq:start}, $\pi_\beta(b) = Q(\pi_\gamma(b), \pi_\beta(\bar{b}))$, which clearly implies $\pi_\beta(b) \in D(\pi_\beta(\bar{b}), \gamma, \xi)$. 

Thus, we have shown $\pi_\beta[A] = D(\pi_\beta(\bar{b}), \gamma, \xi) \in V_\beta \subseteq W_\beta$. Since this holds for every $\beta$ with $\max(\gamma, \xi) < \beta < \lambda$, the claim follows.
\end{proof}

\begin{claim}\label{claim:2}
For every  $\alpha \in S$, every element of $V_\alpha$ is captured below $\alpha$.
\end{claim}
\begin{proof}
Fix $\alpha \in S$,  $H \in [T_\alpha]^{<\kappa}$ and $\beta, \xi < \alpha$. By (Q4), for every $x,y \in H$ with $x \lhd y$ there exists $\gamma_{xy} < \alpha$ such that $x = Q(\pi_{\gamma_{xy}}(x), y)$. If $H$ has at most one element, let $\gamma = 0$, otherwise let 
\[
\gamma = \sup\{\gamma_{xy} : x,y \in H \text{ and } x \lhd y\}.
\]
Since $\alpha$ has cofinality $\kappa$, we have $\gamma < \alpha$. We show that  $D(H, \beta, \xi)$ is captured at $\max(\beta,\gamma, \xi) +1 $. Pick any $\delta$ with $\max(\beta, \gamma, \xi) < \delta < \alpha$. 

By definition, $D(\pi_\delta[H], \beta, \xi) \in V_\delta \subseteq W_\delta$. It suffices to prove that $\pi_{\delta}[D(H, \beta, \xi)] = D(\pi_\delta[H], \beta, \xi)$ and that $\pi_{\delta} \upharpoonright D(H, \beta, \xi)$ is injective.

By Lemma~\ref{lemma:basicquag}(3), for every $x \in H$ and $\bar{y} \trianglelefteq \pi_\beta(x)$ we have $\pi_\delta(Q(\bar{y}, x)) = Q(\bar{y}, \pi_\delta(x))$. Hence, $\pi_{\delta}[D(H, \beta, \xi)] = D(\pi_\delta[H], \beta, \xi)$. 

Now we prove that $\pi_{\delta} \upharpoonright D(H, \beta, \xi)$ is injective. Pick $z_0, z_1 \in D(H, \beta, \xi)$ and assume that $\pi_\delta(z_0) = \pi_\delta(z_1)$. Choose $x_0, x_1 \in H$ such that $z_i = Q(\pi_\beta(z_i), x_i)$ for $i = 0,1$. If $x_0 = x_1$, then $z_0 = z_1$ follows directly. So assume otherwise, say $x_0 \lhd x_1$. Then the following holds
\begin{align*}
z_0 &= Q(\pi_\beta(z_0), x_0)\\&= Q(\pi_\delta(z_0), x_0) \\&= Q(\pi_\delta(z_0), Q(\pi_\delta(x_0), x_1)) \\&= Q(\pi_\delta(z_0), x_1) \\&= Q(\pi_\delta(z_1), x_1) = z_1,
\end{align*}
where the second equality follows from Lemma~\ref{lemma:basicquag}(4), the third one from $\gamma_{x_0 x_1} \le \gamma$ and Lemma~\ref{lemma:basicquag}(4), the fourth one from (Q3), the fifth one from $\pi_\delta(z_0) = \pi_\delta(z_1)$, and the last one from Lemma~\ref{lemma:basicquag}(4). This proves that $\pi_{\delta} \upharpoonright D(H, \beta, \xi)$ is injective.
\end{proof}

\begin{claim}\label{claim:3}
For every $\alpha \in S'$, $\{ \{x\} : x \in T_\alpha\} \subseteq V_\alpha$.
\end{claim}
\begin{proof}
Let $N$ witness $\alpha \in S'$. Pick $x \in T_\alpha$ and choose $\beta < \alpha$. By elementarity of $N$ there is $0 < \xi < \alpha$ such that $\{\pi_\beta(x)\} = X^\beta_\xi.$
 Hence, $\{x\} = D(x, \beta, \xi) \in V_\alpha$.
\end{proof}

The next claim shows that $V_\alpha$ is a $\kappa$-complete generalized Boolean algebra.

\begin{claim}\label{claim:4}
For every $\alpha \in S'$, $V_\alpha$ is closed under ${<} \kappa$-unions and differences.
\end{claim}
\begin{proof}

Fix $\alpha \in S'$ and let $N$ witness this. We now show that  $V_\alpha$ is closed under ${<}\kappa$-unions. The proof of the closure under differences is analogous. Fix some $\eta < \kappa$ and, for every $\iota < \eta$, fix $H_\iota \in [T_\alpha]^{<\kappa}$ and $\beta_\iota, \xi_\iota < \alpha$, towards showing 
\[
\bigcup_{\iota < \eta} D(H_\iota, \beta_\iota, \xi_\iota) \in V_\alpha.
\]

First let $H = \bigcup_{\iota < \eta} H_\iota$. By (Q4), for every $x,y \in H$ with $x \lhd y$, there exists $\gamma_{xy} < \alpha$ such that $\pi_{\gamma_{xy}}(x) \lhd \pi_{\gamma_{xy}}(y)$ and $x = Q(\pi_{\gamma_{xy}}(x), y)$. If $H$ contains at most one element, let $\gamma = 0$, otherwise let 
\[
\gamma = \sup\{\gamma_{xy} : x, y \in H \text{ and } x \lhd y\}.
\]
Let 
\[
\delta = \max\left(\gamma, \sup_{\iota < \eta} \beta_\iota\right) + 1.
\]
 Both $\gamma$ and $\delta$ are $< \alpha$.

By the elementarity of $N$ and its closure under ${<}\kappa$-sequences,  $D(\pi_\delta[H_\iota], \beta_\iota, \xi_\iota)$ belongs to $N$ for each $\iota < \eta$. Hence, the union $\bigcup_{\iota < \eta} D(\pi_\delta[H_\iota], \beta_\iota, \xi_\iota)$ also belongs to $N$. Thus, there exists $0 < \zeta < \alpha$ such that 
\[
X^\delta_\zeta = \bigcup_{\iota < \eta} D(\pi_\delta[H_\iota], \beta_\iota, \xi_\iota).
\]
It suffices to show:
\begin{equation}\label{eq:claim4-0}
D(H, \delta, \zeta) = \bigcup_{\iota < \eta} D(H_\iota, \beta_\iota, \xi_\iota).
\end{equation}

Showing the inclusion $\supseteq$ is easy. Pick $z \in D(x, \beta_\iota, \xi_\iota )$ for some $x \in H_\iota$ and some $\iota < \eta$. Then,
\begin{align*}
z &= Q(\pi_{\beta_\iota}(z), x) \\&= Q(Q(\pi_{\beta_\iota}(z), \pi_\delta(x)), x),
\end{align*}
where the second equality follows from (Q2). Since $Q(\pi_{\beta_\iota}(z), \pi_\delta(x))$ belongs to $X^\delta_\zeta$, we conclude that $z \in D(H, \delta, \zeta)$. So we are left to prove the other inclusion.

Pick $z \in D(H, \delta, \zeta)$. Let $\iota < \eta$ and $x \in H$ be such that $\pi_\delta(z) \in D(\pi_\delta[H_\iota], \beta_\iota, \xi_\iota)$ and $z = Q(\pi_\delta(z), x)$. Thus, $\pi_{\beta_\iota}(z) \in X^{\beta_\iota}_{\xi_\iota}$ and we can pick $y \in H_\iota$ such that $\pi_\delta(z) = Q(\pi_{\beta_\iota}(z), \pi_\delta(y))$. The following holds:

\begin{equation}\label{eq:claim4-1}
\begin{split}
z &= Q(\pi_\delta(z), x)\\ &= Q(Q(\pi_{\beta_\iota}(z), \pi_\delta(y)), x)\\
&= Q(Q(\pi_{\beta_\iota}(z), \pi_\delta(y)), y)\\
&= Q(\pi_{\beta_\iota}(z), y).
\end{split}
\end{equation}
The second equality follows from the choice of $\iota$ and $y$. Now we justify the third equality. If $x = y$, then it holds trivially. So we can suppose either $x \lhd y$ or $y \lhd x$. Since $\gamma \le \delta$, we conclude by definition of $\gamma$ and by Lemma~\ref{lemma:basicquag}(4) that either $x = Q(\pi_\delta(x), y)$ or $y = Q(\pi_\delta(y),x)$. In either case, the third equality holds by (Q3). Finally, the last equality follows from (Q2). 

By \eqref{eq:claim4-1}, we conclude that $z \in D(H_\iota, \beta_\iota, \xi_\iota)$ as we wanted to show. Overall, \eqref{eq:claim4-0} holds and we are done.
\end{proof}

The next key claim describes how $V_\alpha$ interacts with the witnesses $N, Y$ to $(\dagger)$ and, in particular, shows that $V_\alpha$ is closed under intersections with $X^\alpha_0$.

\begin{claim}\label{claim:5}
Let $\alpha \in S'$ and $A \in V_\alpha$. If $\alpha$ satisfies $(\dagger)$ with witnesses $N, Y$, then:
\begin{enumerate}[label={\upshape (\arabic*)}]
\item $Y \cap \pi_{\alpha \lambda}^{-1}(A) \cap N \in N$,
\item $A \cap X^\alpha_0 \in V_\alpha$.
\end{enumerate}
\end{claim}
\begin{proof}

Suppose that $\alpha \in S'$ satisfies $(\dagger)$ with witnesses $N, Y$, and fix $A \in V_\alpha$, towards showing that (1) and (2) hold. Since both $V_\alpha$ and $N$ are closed under ${<}\kappa$-unions, we can assume without loss of generality that $A = D(x, \beta, \xi)$ for some $x \in T_\alpha$ and $\beta, \xi < \alpha$. 

So fix $\beta, \xi < \alpha$ and $x \in T_\alpha$ and  let 
\[
K \coloneqq Y \cap \pi_{\alpha \lambda}^{-1}(D(x, \beta, \xi)) \cap N.
\] 

By Lemma~\ref{lemma:basicquag}, we know that $Y$ has a $\lhd$-upper bound. By elementarity of $N$, we can find such an upper bound in $N$. Hence, let $\bar{b} \in T_{\lambda} \cap N$ be such that $Y \subseteq \{b \in T_{\lambda} : b \lhd \bar{b}\}$.

There are three cases: either $x = \pi_\alpha(\bar{b})$ or $x \lhd \pi_\alpha(\bar{b})$ or $\pi_\alpha(\bar{b}) \lhd x$. We only consider the latter case, as the other two are analogous.  By (Q4), there exists $\gamma < \alpha$ such that 
\[
\pi_\gamma(\bar{b}) \lhd \pi_\gamma(x) \text{ and } \pi_\alpha(\bar{b}) = Q(\pi_\gamma(\bar{b}),  x).
\]
Let $\delta = \max(\beta, \gamma) + 1$ and
\begin{align*}
I &\coloneqq \big\{z \in T_\delta : z \lhd \pi_\delta (\bar{b}) \text{ and } Q(z, \bar{b}) \in Y\big\},\\X &\coloneqq I \cap D(\pi_\delta(x), \beta, \xi).
\end{align*}
Since $X$ belongs to $N$ by elementarity (using clause (4) of $(\dagger)$ if $\xi = 0$), we can fix $0 < \zeta < \alpha$ such that $X^\delta_\zeta = X$. Both (1) and (2) quickly follow once we show $K = D(\bar{b}, \delta, \zeta)$.

First pick $b \in D(\bar{b}, \delta, \zeta)$, towards showing that $b \in Y \cap N$ and $\pi_\alpha(b) \in D(x, \beta, \xi)$. By the choice of $\zeta$, we know that
\begin{enumerate}[label=(\alph*)]
\item $b = Q(\pi_\delta(b), \bar{b})$,
\item $\pi_\delta(b) \in I$,
\item $\pi_\beta(b) \in X^\beta_\xi$,
\item $\pi_\beta(b) \trianglelefteq \pi_\beta(x)$,
\item $\pi_\delta(b) = Q(\pi_\beta(b), \pi_\delta(x))$.
\end{enumerate}
It follows directly from (a) and (b) that $b \in Y$. Moreover, as $\pi_\delta(b) \in N$---indeed, $T \upharpoonright \alpha \subseteq N$---we also have $b \in N$ by (a) and the elementarity of $N$. Hence, $b \in Y \cap N$, and we are left to show that $\pi_\alpha(b) \in D(x, \beta, \xi)$. The following holds:
\begin{align*}
\pi_\alpha(b) &= \pi_\alpha(Q(\pi_\delta(b), \bar{b}))\\&= Q(\pi_\delta(b), \pi_\alpha(\bar{b}))\\&= Q(\pi_\delta(b), Q(\pi_\delta(\bar{b}), x)) \\&= Q(\pi_\delta(b), x) \\&= Q(Q(\pi_\beta(b), \pi_\delta(x)), x) \\&= Q(\pi_\beta(b), x) \in D(x, \beta, \xi).
\end{align*} 
The first equality follows from (a); the second one from Lemma~\ref{lemma:basicquag}(3); the third one from the choice of $\gamma$ and from Lemma~\ref{lemma:basicquag}(4); the fourth one from (Q3) and from $\pi_\delta(b) \lhd \pi_\delta(\bar{b})$, which holds by (b); the fifth one from (e); the last one from (Q2), and the final membership relation holds by (c) and (d).

So we have shown $D(\bar{b}, \delta, \zeta) \subseteq K$. Now let us pick $b \in Y \cap N$ such that $\pi_\alpha(b) \in D(x, \beta, \xi)$ and let us prove $b \in D(\bar{b}, \delta, \zeta)$. From our assumption we get that 
\begin{enumerate}[label=(\roman*)]
\item $\pi_\beta(b) \trianglelefteq \pi_\beta(x)$,
\item $\pi_\beta(b) \in X^\beta_\xi$,
\item $\pi_\alpha(b) = Q(\pi_\beta(b), x)$.
\end{enumerate}
Moreover, the following holds:
\begin{subclaim}\label{claim:5-1}
$\pi_\delta(b) \lhd \pi_\delta(\bar{b})$.
\end{subclaim}
\begin{proof}
Note first that $\pi_\alpha(b) \lhd \pi_\alpha(\bar{b})$. Indeed, using (Q4), the elementarity of $N$, and the fact that both $b$ and $\bar{b}$ belong to $N$, there exists $\eta < \alpha$ such that $\pi_\eta(b) \lhd \pi_\eta(\bar{b})$ and $b = Q(\pi_\eta(b), \bar{b})$. Then, $\pi_\alpha(b) \lhd \pi_\alpha(\bar{b})$ follows from Lemma~\ref{lemma:basicquag}(4). 

Now, by (iii), by the choice of $\gamma$, and by Lemma~\ref{lemma:basicquag}(4), we have that $\pi_\alpha(b) = Q(\pi_\delta(b), x)$ and $\pi_\alpha(\bar{b}) = Q(\pi_\delta(\bar{b}), x)$. In particular, since the map $Q(\cdot, x) \upharpoonright T_\delta$ is an order-embedding  (Lemma~\ref{lemma:basicquag}(2)), we get that $\pi_\delta(b) \lhd \pi_\delta(\bar{b})$, as we wanted to show.
\end{proof}

By (i)-(iii) and Lemma~\ref{lemma:basicquag}(3), we have $\pi_\delta(b) \in D(\pi_\delta(x), \beta, \xi)$. Moreover, by hypothesis, $b \in Y$ and, by Claim~\ref{claim:5-1}, $\pi_\delta(b) \lhd \pi_\delta(\bar{b})$. Hence, once we prove $b = Q(\pi_\delta(b), \bar{b})$, the fact that $b \in Y$ will imply $\pi_\delta(b) \in I$, and therefore $\pi_\delta(b) \in X = X^\delta_\zeta$. Thus, it remains to show $b = Q(\pi_\delta(b), \bar{b})$. 

As $b \in N$ by hypothesis and $\pi_{\alpha \lambda} \upharpoonright N$ is injective by elementarity of $N$, we just need to show that $\pi_\alpha(b) = \pi_\alpha(Q(\pi_\delta(b), \bar{b}))$---note that $Q(\pi_\delta(b), \bar{b})$ belongs to $N$ by elementarity, since both $\bar{b}$ and $\pi_\delta(b)$ belong to $N$. We have
\begin{align*}
\pi_\alpha(Q(\pi_\delta(b), \bar{b})) &= Q(\pi_\delta(b), \pi_\alpha(\bar{b}))\\&= Q(\pi_\delta(b), Q(\pi_\delta(\bar{b}), x)) \\&= Q(\pi_\delta(b), x) \\&= \pi_\alpha(b),
\end{align*}
where the first equality follows from Lemma~\ref{lemma:basicquag}(3); the second one from the choice of $\gamma$ and Lemma~\ref{lemma:basicquag}(4); the third one holds by (Q3) and Claim~\ref{claim:5-1}, and the last one by (iii) and Lemma~\ref{lemma:basicquag}(4). 

Overall, we have shown $K = D(\bar{b}, \delta, \zeta)$. Now, $K \in N$ clearly follows, since $D(\bar{b}, \delta, \zeta)$ belongs to $N$ by elementarity of $N$.  Moreover, it follows that $D(x, \beta, \xi) \cap X^\alpha_0 \in V_\alpha$. Indeed, the following holds:
\begin{align*}
D(x, \beta, \xi) \cap X^\alpha_0 &= D(x, \beta, \xi) \cap \pi_\alpha[Y \cap N]\\&= \pi_\alpha[K]\\&= \pi_\alpha[D(\bar{b}, \delta , \zeta)] \\&= D(\pi_\alpha(\bar{b}), \delta, \zeta) \in V_\alpha,
\end{align*} 
where the first equality holds by definition of $X^\alpha_0$; the second equality directly follows from the definition of $K$; in the third equality we used $K = D(\bar{b}, \delta, \zeta)$, and the last one holds by Lemma~\ref{lemma:basicquag}(3) and the definition of $D$. This completes the proof.
\end{proof}

Now, given an $\alpha < \lambda$, let $V^{\complement}_\alpha = \{T_\alpha \setminus A : A \in V_\alpha\}$ and $\hat{V}_\alpha = V_\alpha \cup V_\alpha^\complement$.

\begin{claim}\label{claim:6}
For every $\alpha \in S'$, 
\begin{multline*}
W_\alpha = V_\alpha \cup \big\{(A_0 \cap X^\alpha_0) \cup (A_1 \cap (T_\alpha \setminus X^\alpha_0)) : A_0, A_1 \in \hat{V}_\alpha\\ \text{and } \big(A_0 \in V_\alpha^\complement \text{ or } A_1 \in V_\alpha^\complement\big)\big\}.
\end{multline*}
\end{claim}
\begin{proof}

It quickly follows from Claim~\ref{claim:4} that $\hat{V}_\alpha$ is a $\kappa$-complete Boolean subalgebra of $\mathcal{P}(T_\alpha)$. This fact, together with Claim~\ref{claim:3}, implies that $W_\alpha$ is the Boolean subalgebra of $\mathcal{P}(T_\alpha)$ generated by $\hat{V}_\alpha \cup \{X^\alpha_0\}$. Hence
\begin{equation}\label{eq:closure}
W_\alpha = \{(A_0 \cap X^\alpha_0) \cup (A_1 \cap (T_\alpha \setminus X^\alpha_0)): A_0, A_1 \in \hat{V}_\alpha\}.
\end{equation}

If $\alpha$ does not satisfy $(\dagger)$, then $X^\alpha_0 = \emptyset$, and thus our claim holds trivially---in particular, $W_\alpha = \hat{V}_\alpha$. Hence we can assume that $\alpha$ satisfies $(\dagger)$. By \eqref{eq:closure} and by the closure of $V_\alpha$ under finite unions and differences, it suffices to prove that $A \cap X^\alpha_0 \in V_\alpha$ for every $A \in V_\alpha$. But this holds by Claim~\ref{claim:5}(2).
\end{proof}

For every $\alpha < \lambda$, let $R_\alpha$ denote the set $\{A \in W_\alpha : A \text{ is captured below } \alpha\}$.

\begin{claim}\label{claim:7}
Let $\alpha \in S'$. Then the following hold:
\begin{itemize}
\item if $\alpha$ satisfies $(\dagger)$, then $A \not\in R_\alpha$ for every $A \subseteq T_\alpha$ such that, for some $B \in V_\alpha$, either $A \supseteq X^\alpha_0 \setminus B$ or $A \supseteq T_\alpha \setminus (X^\alpha_0 \cup B)$.
\item otherwise, $A \not\in R_\alpha$ for every $A \subseteq T_\alpha$ such that $A \supseteq T_\alpha \setminus B$ for some $B \in V_\alpha$.
\end{itemize}
\end{claim}
\begin{proof}
Fix $\alpha \in S'$ and first suppose that $\alpha$ does not satisfy $(\dagger)$. Let $N$ be a witness to $\alpha \in S'$. It suffices to show that $\pi_\beta \upharpoonright (T_\alpha \setminus B)$ is not injective for every $B \in V_\alpha$ and $\beta < \alpha$. Assume otherwise towards a contradiction and pick witnesses $B \in V_\alpha$ and $\beta < \alpha$.  

By Claim~\ref{claim:2}, $B$ is captured below $\alpha$. Hence, in particular, there exists some $\gamma \ge \beta$ such that $\pi_\gamma \upharpoonright B$ is injective. Since both $\pi_\gamma \upharpoonright B$ and $\pi_\gamma \upharpoonright (T_\alpha \setminus B)$ are injective, we conclude that the map $\pi_{\gamma \alpha}$ has fibers of cardinality at most $2$. Since the map $\pi_{\alpha \lambda} \upharpoonright N$ is injective by elementarity of $N$, we conclude that also the fibers of $\pi_{\gamma \lambda} \upharpoonright N$ have cardinality at most $2$. By elementarity of $N$ again, $\pi_{\gamma \lambda}$ must have fibers of cardinality at most $2$, which is absurd, since $T_\gamma$ has cardinality at most $\kappa$, while $T_\lambda$ has cardinality $\lambda^+$.

Now suppose that $\alpha$ satisfies $(\dagger)$ and let $N, Y$ be the witnesses. Fix also some $B \in V_\alpha$. It suffices to prove that, for every $\beta < \alpha$, neither $\pi_{\beta} \upharpoonright (X^\alpha_0 \setminus B)$ nor  $\pi_{\beta} \upharpoonright (T_\alpha \setminus (X^\alpha_0 \cup  B))$ is injective.

Fix $\beta < \alpha$ and suppose towards a contradiction that $\pi_{\beta} \upharpoonright (X^\alpha_0 \setminus B)$ is injective. By Claim~\ref{claim:5}(1), we know that 
\[
K \coloneqq Y \cap \pi_{\alpha \lambda}^{-1}(B) \cap N \in N.
\]
Then, note that
\begin{equation}\label{eq:7-1}
\pi_\beta \upharpoonright ((Y \setminus K) \cap N) = (\pi_\beta \upharpoonright (X^\alpha_0 \setminus B)) \circ (\pi_\alpha \upharpoonright ((Y \setminus K) \cap N)).
\end{equation}
Indeed, $(Y \setminus K) \cap N = (Y \cap N) \setminus \pi_{\alpha \lambda}^{-1}(B)$ and $\pi_\alpha[(Y \setminus K) \cap N] = \pi_\alpha[Y \cap N] \setminus B = X^\alpha_0 \setminus B$, where the last equality holds since $X^\alpha_0 = \pi_\alpha[Y \cap N]$.

Since $\pi_{\alpha \lambda} \upharpoonright N$ is injective and $\pi_{\beta} \upharpoonright (X^\alpha_0 \setminus B)$ is injective by our (absurd) assumption, we conclude by \eqref{eq:7-1} that $\pi_\beta \upharpoonright ((Y \setminus K) \cap N)$ is also injective. But by elementarity of $N$, we get that $\pi_\beta \upharpoonright (Y \setminus K)$ is injective. 

Now, by the choice of $Y$, we know that $Y$ has cardinality $\lambda$. On the other hand, we claim that $K$ has cardinality at most $\kappa$. Indeed, $B$, being a subset of $T_\alpha$, has cardinality at most $\kappa$. Therefore, since the map $\pi_{\alpha \lambda} \upharpoonright N$ is injective, also the set $\pi_{\alpha \lambda}^{-1}(B) \cap N$ has cardinality at most $\kappa$. But since $K \subseteq \pi_{\alpha \lambda}^{-1}(B) \cap N$, we conclude that $K$ has cardinality at most $\kappa$. 

Hence, $Y \setminus K$ has cardinality $\lambda$, and $\pi_\beta \upharpoonright (Y \setminus K)$ cannot be injective, since $T_\beta$ has cardinality at most $\kappa$. Contradiction.

Now we argue that $\pi_{\beta} \upharpoonright (T_\alpha \setminus (X^\alpha_0 \cup  B))$  is not injective. We omit some details since the proof is very similar to the argument given in the case in which $\alpha$ does not satisfy $(\dagger)$. Suppose towards a contradiction that $\pi_{\beta} \upharpoonright (T_\alpha \setminus (X^\alpha_0 \cup  B))$ is injective for some $\beta < \alpha$. Then, there exists $\gamma$ with $\beta \le \gamma < \alpha$ such that $\pi_\gamma \upharpoonright B$ is injective. In particular, the map $\pi_\gamma \upharpoonright (T_\alpha \setminus X^\alpha_0)$ has fibers of cardinality at most $2$. Since the map $\pi_{\alpha \lambda} \upharpoonright N$ is injective and $X^\alpha_0 = \pi_\alpha[Y \cap N]$, we conclude that the fibers of the map $\pi_\gamma \upharpoonright ((T_\lambda \setminus Y) \cap N)$ also have cardinality at most $2$. By elementarity of $N$, also the fibers of $\pi_\gamma \upharpoonright (T_\lambda \setminus Y)$ have cardinality at most $2$, which is a contradiction, since $T_\gamma$ has cardinality at most $\kappa$ while $T_\lambda \setminus Y$ has cardinality $\lambda^+$.
\end{proof}

\begin{claim}\label{claim:8}
For every $\alpha \in S'$, the following hold:
\begin{enumerate}[label={\upshape (\arabic*)}]
\item $R_\alpha = V_\alpha$,
\item $W_\alpha \setminus V_\alpha$ is upward closed, meaning that for every $A, B \in W_\alpha$ with $A \subseteq B$, if $A \not\in V_\alpha$, then $B \not\in V_\alpha$.
\end{enumerate}
\end{claim}
\begin{proof}
Fix $\alpha \in S'$. First suppose that $\alpha$ does not satisfy $(\dagger)$. Then, $X^\alpha_0 = \emptyset$, and, by Claim~\ref{claim:6}, $W_\alpha = \hat{V}_\alpha =  V_\alpha \cup V_\alpha^\complement$. Hence,
\begin{equation}\label{eq:es1}
W_\alpha \setminus V_\alpha^\complement \subseteq V_\alpha \subseteq R_\alpha \subseteq W_\alpha \setminus \{A \in W_\alpha : \exists B \in V_\alpha^\complement \ (A \supseteq B)\} \subseteq W_\alpha \setminus V_\alpha^\complement,
\end{equation}
where the first inclusion follows from $W_\alpha = V_\alpha \cup V_\alpha^\complement$; the second one follows from Claim~\ref{claim:2}, the third one holds by Claim~\ref{claim:7}, and the last one is trivial. Overall, \eqref{eq:es1} is a chain of equalities. Hence, $V_\alpha = R_\alpha$, and $W_\alpha \setminus V_\alpha = \{A \in W_\alpha : \exists B \in V_\alpha^\complement \ (A \supseteq B)\}$, which  is clearly  upward closed.

Now assume that $\alpha$ satisfies $(\dagger)$. Let
\begin{multline*}
P_\alpha \coloneqq \big\{(A_0 \cap X^\alpha_0) \cup (A_1 \cap (T_\alpha \setminus X^\alpha_0)) : A_0, A_1 \in \hat{V}_\alpha\\\text{and } \big(A_0 \in V_\alpha^\complement \text{ or } A_1 \in V_\alpha^\complement\big)\big\}
\end{multline*}
and let 
\[
Q_\alpha \coloneqq \big\{A \in W_\alpha : \exists B \in V_\alpha \ (A \supseteq X^\alpha_0 \setminus B \text{ or } A \supseteq T_\alpha \setminus (X^\alpha_0 \cup B))\big\}.
\]
Then, the following holds:
\begin{equation}\label{eq:es2}
W_\alpha \setminus P_\alpha \subseteq V_\alpha \subseteq R_\alpha \subseteq W_\alpha \setminus Q_\alpha \subseteq W_\alpha \setminus P_\alpha,
\end{equation}
where the first inclusion follows from Claim~\ref{claim:6}, the second one from Claim~\ref{claim:2}, the third one from Claim~\ref{claim:7}, and the last one holds since $P_\alpha \subseteq Q_\alpha$. So \eqref{eq:es2} is a chain of equalities. It follows that $R_\alpha = V_\alpha$. Moreover, we also have $W_\alpha \setminus V_\alpha = Q_\alpha$. Since $Q_\alpha$ is clearly upward closed, (2) follows as well.
\end{proof}

\begin{claim}\label{claim:9}
For every $\alpha \in S'$, $R_\alpha$ is closed under ${<}\kappa$-unions and under taking subsets that belong to $W_\alpha$.
\end{claim}
\begin{proof}
Let $\alpha \in S'$.  The closure of $R_\alpha$ under ${<}\kappa$-unions directly follows from  Claim~\ref{claim:4} and Claim~\ref{claim:8}(1). The closure of $R_\alpha$ under taking subsets in $W_\alpha$ directly follows from Claim~\ref{claim:8}(1) and (2).
\end{proof}

By Claim~\ref{claim:9}, we conclude that clause (3) of the principle $W^*_\lambda(\lambda^+)$ is satisfied. We are left to prove the fourth clause.  Fix $f:\lambda \rightarrow ([T_{\lambda}]^{\le\kappa})^2$ such that 
\[
Y \coloneqq \bigcup_{\eta < \lambda} (f_0(\eta) \cup f_1(\eta))
\] 
has cardinality $\lambda$. Let $\langle b_\mu : \mu < \lambda\rangle$ enumerate $Y$ without repetitions. Define
\begin{align*}
\Omega^0 &\coloneqq \{(\mu, x) : \mu < \lambda, x \in T \text{ and } x < b_\mu\},\\
\Omega^1 &\coloneqq \{(\eta, \mu, i) : \eta,\mu < \lambda \text{ and } i \in \{0,1\} \text{ and } b_\mu \in f_i(\eta)\},\\
\Omega &\coloneqq \Omega^0 \cup \Omega^1.
\end{align*}

Note that $\Omega \subseteq H(\lambda)$. By the choice of the sequence $\langle \Omega_\alpha : \alpha < \lambda\rangle$, there are stationarily many $\alpha \in S$ for which there exists an elementary submodel $N \prec H(\lambda^{++})$ satisfying the following conditions: 
\begin{enumerate}[label=(\alph*)]
\item ${}^{<\kappa}N \subseteq N$,
\item $\alpha = \lambda \cap N$,
\item $T, <, \lhd, Q, \langle X^\beta_\xi : \beta, \xi < \lambda\rangle, f, \langle b_\mu : \mu < \lambda\rangle\in N$,
\item $\Omega_\alpha = \Omega \cap N$.
\end{enumerate} 
For the remainder of the proof, we fix one such $\alpha$ and $N$. Note that $\alpha \in S'$. Moreover, since $\Omega^0, \Omega^1$, and $\Omega$ are definable from parameters belonging to $N$, elementarity yields that all three of them belong to $N$.
\pagebreak
\begin{claim}\label{claim:10}
$\alpha$ satisfies $(\dagger)$ and $x^\alpha_\mu = \pi_\alpha(b_\mu)$ for every $\mu < \alpha$. In particular, $X^\alpha_0 = \pi_\alpha[Y \cap N]$.
\end{claim}
\begin{proof}
Let us show that $N$ and $Y$ witness that $\alpha$ satisfies $(\dagger)$. All conditions of $(\dagger)$ are clearly met except for (6). So we have to prove that $\{x^\alpha_\mu : \mu < \alpha\} = \pi_\alpha[Y \cap N]$.

By $\Omega^0 \in N$ and by elementarity of $N$,
\[
\Omega^0 \cap N = \{(\mu, x) : \mu < \alpha, x \in T \upharpoonright \alpha \text{ and } x < b_\mu\}.
\]
Since $\Omega_\alpha = \Omega \cap N$, it follows that, for every $\mu < \alpha$,
\[
\{y \in T \upharpoonright \alpha : (\mu, y) \in \Omega_\alpha\} = \{y : y < \pi_\alpha(b_\mu)\}.
\]
Hence, by the definition of $x^\alpha_\mu$, we have $x^\alpha_\mu = \pi_\alpha(b_\mu)$ for every $\mu < \alpha$. So 
\begin{equation}\label{eq:dagger}
\{x^\alpha_\mu : \mu < \alpha\} = \{\pi_\alpha(b_\mu) : \mu < \alpha\} = \pi_\alpha[Y \cap N],
\end{equation}
where the last equality directly follows from $\{b_\mu : \mu < \alpha\} = Y \cap N$.  Thus, clause (6) of $(\dagger)$ is satisfied and, moreover, $x^\alpha_\mu = \pi_\alpha(b_\mu)$ for every $\mu < \alpha$.
\end{proof}

\begin{claim}\label{claim:11}
$\{b_\mu : \mu < \alpha\}$ is captured at $\alpha$. 
\end{claim}
\begin{proof}
By Claim~\ref{claim:10}, $\alpha$ satisfies $(\dagger)$ and 
\[
X^\alpha_0 = \{\pi_\alpha(b_\mu) : \mu < \alpha\} = \pi_\alpha[Y \cap N] \in W_\alpha.
\] 
Moreover, since the map $\pi_{\alpha \lambda} \upharpoonright  N$ is injective, the maps $\pi_\beta \upharpoonright (Y \cap N)$ are also injective for every $\beta$ with $\alpha \le \beta < \lambda$. Hence, since we already noted that $X^\alpha_0 \in W_\alpha$,  it suffices to show that $\pi_\beta[Y \cap N] \in V_\beta$ for every $\beta$ with $\alpha < \beta < \lambda$. Pick such a $\beta$.

Let $\bar{b}$ be an $\lhd$-upper bound of $Y$ in $N$. We claim that 
\[
\pi_\beta[Y \cap N] = D(\pi_\beta(\bar{b}), \alpha, 0).
\]
 For every $b \in Y \cap N$,
\begin{equation}\label{eq:capture}
\begin{split}
\pi_\beta(b) &= \pi_\beta(Q(\pi_\alpha(b), \bar{b}))\\&= Q(\pi_\alpha(b), \pi_\beta(\bar{b})).
\end{split}
\end{equation}
The first equality holds because, by elementarity of $N$ and $b \lhd \bar{b}$ and (Q4), there exists $\gamma < \alpha$ such that $b = Q(\pi_\gamma(b), \bar{b})$; then, by Lemma~\ref{lemma:basicquag}(4), we conclude $b = Q(\pi_\alpha(b), \bar{b})$. The second equality follows from Lemma~\ref{lemma:basicquag}(3). 

Now pick $z \in \pi_\beta[Y \cap N]$. Let $b \in Y \cap N$ be such that $z = \pi_\beta(b)$. By \eqref{eq:capture}, $z = Q(\pi_\alpha(b), \pi_\beta(\bar{b}))$. As $\pi_\alpha(b) \in X^\alpha_0$, we conclude that $z \in D(\pi_\beta(\bar{b}), \alpha, 0)$.

Conversely, suppose $z \in D(\pi_\beta(\bar{b}), \alpha, 0)$. This means that $\pi_\alpha(z) \trianglelefteq \pi_\alpha(\bar{b})$, and $\pi_\alpha(z) \in X^\alpha_0$, and $z = Q(\pi_\alpha(z), \pi_\beta(\bar{b}))$. As $X^\alpha_0 = \pi_\alpha[Y \cap N]$ by Claim~\ref{claim:10}, we can pick $b \in Y \cap N$ such that $\pi_\alpha(z) = \pi_\alpha(b)$. Then, $z = Q(\pi_\alpha(b), \pi_\beta(\bar{b}))$. From \eqref{eq:capture}, we conclude $z = \pi_\beta(b)$ and hence $z \in \pi_\beta[Y \cap N]$. 

Overall, we have shown $\pi_\beta[Y \cap N] = D(\pi_\beta(\bar{b}), \alpha, 0) \in V_\beta$, as desired.
\end{proof}

\begin{claim}\label{claim:12}
Given $A \in R_\alpha$, if $A \subseteq \{\pi_\alpha(b_\mu) : \mu < \alpha\}$, then 
\[
\sup \{\mu < \alpha : \pi_\alpha(b_\mu) \in A\} < \alpha.
\]
\end{claim}
\begin{proof}
Fix $A \in R_\alpha$ with $A \subseteq \{\pi_\alpha(b_\mu) : \mu < \alpha\}$. By Claim~\ref{claim:8}(1), $A \in V_\alpha$.  By Claim~\ref{claim:5}(1), we know that 
\[
K \coloneqq Y \cap \pi_{\alpha \lambda}^{-1}(A) \cap N \in N.
\]
Then,
\begin{equation}\label{eq:13}
\{\mu < \alpha : \pi_\alpha(b_\mu) \in A\} = \{\mu < \lambda : b_\mu \in K\} \in N,
\end{equation}
where the equality follows from the definition of $K$ together with Claim~\ref{claim:10}, and the final membership in $N$ holds because both the enumeration of $Y$ and $K$ belong to $N$. 

Now, as we already noted in the proof of Claim~\ref{claim:7}, the set $K$ has cardinality at most $\kappa$, since $A$ has cardinality at most $\kappa$ and $\pi_{\alpha \lambda} \upharpoonright N$ is injective. Hence, $ \sup \{\mu < \lambda : b_\mu \in K\} < \lambda$. By elementarity of $N$, $ \sup \{\mu < \lambda : b_\mu \in K\} < \alpha$. This inequality together with \eqref{eq:13} proves our claim.
\end{proof}

\begin{claim}\label{claim:13}
$A_\alpha = \{(\pi_\alpha[f_0(\beta)], \pi_\alpha[f_1(\beta)]) : \beta < \alpha\}$.
\end{claim}
\begin{proof}
Since $f$ belongs to $N$, and $\kappa \cup \{\kappa\} \subseteq N$, we have
\[
f_i(\eta) \subseteq \{b_\mu : \mu < \alpha\} = Y \cap N
\]
 for every $i\in\{0,1\}$ and $\eta < \alpha$. Moreover, since $\Omega^1 \in N$, the elementarity of $N$ yields
\[
\Omega^1 \cap N = \{(\eta, \mu, i) : \eta, \mu < \alpha \text{ and } i\in \{0,1\} \text{ and } b_\mu \in f_i(\eta)  \}.
\]
From $\Omega_\alpha = \Omega \cap N$ and Claim~\ref{claim:10}, we have that, for every $\eta < \alpha$ and $i \in \{0,1\}$, 
\[
\{x^\alpha_\mu : (\eta, \mu, i) \in \Omega_\alpha\} = \pi_\alpha[f_i(\eta)].
\]
So in order to prove our claim it suffices to show that $\pi_\alpha[f_i(\beta)] \in W_\alpha$ for every $\beta < \alpha$ and $i = 0,1$. But by Claim~\ref{claim:1} and the elementarity of $N$, we know that for every $X \in [T_{\lambda}]^{\le\kappa} \cap N$ there exists $\gamma < \alpha$ such that $\pi_\delta[X] \in W_\delta$ for every $\delta$ with $\gamma \le \delta < \lambda$---note that the sequence $\langle W_\beta : \beta < \lambda\rangle$ belongs to $N$, being definable from parameters in $N$.  Hence, since $f_i(\beta) \in [T_{\lambda}]^{\le\kappa} \cap N$, we conclude that $\pi_\alpha[f_i(\beta)] \in W_\alpha$, as desired.
\end{proof}

Claims~\ref{claim:11}, \ref{claim:12}, and \ref{claim:13} establish clauses (4a), (4b), and (4c), respectively, for every \(\alpha\) in the stationary subset of $S'$ considered above. Hence clause (4) of $W^*_\lambda(\lambda^+)$ holds, completing the proof.
\end{proof}
\endgroup
\section{Questions}\label{sec:questions}

Under $\mathsf{V=L}$, for every uncountable regular cardinal $\kappa$, there exists a $\kappa$-Suslin tree if and only if $\kappa$ is not weakly compact \cites{MR309729, MR3692231}. The natural question is whether we can recover the same characterization in terms of Suslin $(\kappa, \kappa^+)$-forests instead. Since every Suslin $(\kappa, \kappa^+)$-forest naturally gives rise to $\kappa$-Suslin trees (see the introduction), the question is really the following:
\begin{question}
Under $\mathsf{V=L}$, is there a Suslin $(\kappa, \kappa^+)$-forest for every uncountable regular cardinal $\kappa$ which is not weakly compact?
\end{question}

We can strengthen the previous question by asking the forests to satisfy the following strong form of coherence. A $(\kappa, \lambda)$-forest $F$ is \emph{$\omega$-coherent} if for every  $f,g \in F$, the set $\{\alpha \in \mathrm{dom}(f) \cap \mathrm{dom}(g) : f(\alpha) \neq g(\alpha)\}$ is finite \cite{MR3326048}. It is known that, under $\mathsf{V=L}$, there exists a binary $\omega$-coherent $\kappa$-Suslin tree\footnote{In the case of trees, finite-difference coherence is usually referred to simply as coherence.} for every uncountable regular non-weakly compact cardinal $\kappa$  \cite{MR3692231}.

\begin{question}
Under $\mathsf{V=L}$, is there an $\omega$-coherent Suslin $(\kappa, \kappa^+)$-forest for every uncountable regular cardinal $\kappa$ which is not weakly compact?
\end{question}

A separate problem is to recover a proof of Laver's result that the existence of a Suslin $(\omega_1, \omega_2)$-forest follows from Silver's principle $W_{\omega_1}(\omega_2)$ together with $\diamondsuit$. We expect a construction using only $\diamondsuit + W_{\omega_1}(\omega_2)$ not to produce the coherence obtained in Eskew's construction.

\begin{problem}
Prove Laver's result that the existence of a Suslin $(\omega_1, \omega_2)$-forest follows from $\diamondsuit + W_{\omega_1}(\omega_2)$.
\end{problem}

Finally, we do not know whether the existence of a $(\omega_1, \mathrm{Lim}(\omega_1))$-quagmire is genuinely weaker than the existence of a $(\omega_1, 1)$-morass. Note that Komj\'{a}th \cite{MR877858} constructed, assuming the existence of a Mahlo cardinal, a model of $\mathsf{ZFC}$ in which there is an $\omega_1$-quagmire (i.e., a $(\omega_1, \emptyset)$-quagmire in our notation) but no $(\omega_1, 1)$-morass. 
\begin{question}
Is it consistent that there is a $(\omega_1, \mathrm{Lim}(\omega_1))$-quagmire but no $(\omega_1, 1)$-morass?
\end{question}

\printbibliography

\end{document}